\documentclass[preprint,12pt]{elsarticle}

\usepackage{graphicx}
\usepackage{amssymb}
\usepackage{amsmath}
\usepackage{mathtools}
\usepackage{seqsplit}
\usepackage[utf8]{inputenc}
\usepackage{float}

\usepackage{lineno}
\usepackage{amsthm}
\usepackage{dsfont}
\usepackage{hyperref}
\usepackage{tikz}
\usepackage[left=1.5cm, right=1.5cm, top=1.5cm, bottom=1.5cm]{geometry}
\usetikzlibrary{positioning}
  \def\gn#1#2{{$\href{http://groupnames.org/\#?#1}{#2}$}}
\def\gn#1#2{$#2$}  
\tikzset{sgplattice/.style={inner sep=1pt,norm/.style={red!50!blue},char/.style={blue!50!black},
  lin/.style={black!50}},cnj/.style={black!50,yshift=-2.5pt,left=-1pt of #1,scale=0.5,fill=white}}

\newcommand{\mychi}{\raisebox{0pt}[1ex][1ex]{$\chi$}}

\newcommand{\SL}{\operatorname{SL}}
\newcommand{\TL}{\operatorname{TL}}

\newcommand{\IdGroup}{\operatorname{IdGroup}}

\newcommand{\id}{\operatorname{id}}

\newcommand{\Diff}{\operatorname{Diff}}

\newcommand{\SmallGroup}{\operatorname{SmallGroup}}
\newcommand{\Mod}[1]{\ (\mathrm{mod}\ #1)}

\newcommand{\Ker}{\operatorname{Ker}}

\newcommand\blfootnote[1]{%
  \begingroup
  \renewcommand\thefootnote{}\footnote{#1}%
  \addtocounter{footnote}{-1}%
  \endgroup
}
\newtheorem{theorem}{Theorem}[section]

\newtheorem{lemma}[theorem]{Lemma}

\newtheorem{proposition}[theorem]{Proposition}

\newtheorem*{remark*}{Remark}

\newtheorem*{theorem*}{Theorem}
\newtheorem*{conjecture*}{Conjecture}
\newtheorem*{question*}{Question}

\numberwithin{table}{section}

\journal{Journal Name}

\makeatletter
\def\ps@pprintTitle{%
 \let\@oddhead\@empty
 \let\@evenhead\@empty
 \def\@oddfoot{}%
 \let\@evenfoot\@oddfoot}
\makeatother

\begin{document}

\author{Piotr Mizerka}

\nolinenumbers
\setcounter{section}{-1}
\begin{abstract}
There are four groups $G$ fitting into a short exact sequence
$
1\rightarrow\SL(2,5)\rightarrow G\rightarrow C_2\rightarrow 1,
$
where $\SL(2,5)$ is the special linear group of $(2\times 2)$-matrices with entries in the field of five elements. Except for the direct product of $\SL(2,5)$ and $C_2$, there are two other semidirect products of these two groups and just one non-semidirect product $\SL(2,5).C_2$, considered in this paper. It is known that each finite nonsolvable group can act on spheres with arbitrary positive number of fixed points. Clearly, $\SL(2,5).C_2$ is a nonsolvable group. Moreover, it turns out that $\SL(2,5).C_2$ possesses a free representation and as such, can potentially act pseudofreely with nonempty fixed point set on manifolds of arbitrarily large dimension. We prove that $\SL(2,5).C_2$ cannot act effectively with odd number of fixed points on low-dimensional spheres. In the special case of effective one fixed point actions, we are able to exclude a wider class of spheres. Moreover, we prove that specific pseudofree one fixed point actions of $\SL(2,5).C_2$ on spheres do not exist.\blfootnote{$2020$ Mathematics Subject Classification. Primary $57S25$; Secondary $55M35$.}
\end{abstract}

\begin{frontmatter}
\title{\textbf{\centerline{Exclusions of smooth actions on spheres}\newline
of the non-split extension of $C_2$ by $\SL(2,5)$}}
\end{frontmatter}
\section{Introduction}\label{section:intro}
A smooth action of finite group $G$ on a smooth manifold $M$ is a homomorphism $\varphi:G\rightarrow \Diff(M)$, where $\Diff(M)$ denotes the group of diffeomorphisms of $M$. All groups appearing in this paper are assumed to be finite and actions to be smooth. This article concerns exclusions of actions on spheres with exactly one fixed point or with odd number of fixed points. We refer to such actions as \emph{one fixed point actions} or \emph{odd fixed point actions} respectively.

In \cite[p. 55]{Montgomery1946} Montgomery and Samelson wrote the following: \emph{"If a compact Lie group $G$ acts smoothly on the $n$-sphere $S^n$ in such
a way as to have one stationary point, it is likely that there must be a second
stationary point. This appears difficult to prove."}. It is already known which groups admit one fixed point actions on spheres. It turns out that such groups are precisely the groups admitting fixed point free actions on disks. The characterization of groups admitting fixed point free actions on disks was provided by Oliver \cite{Oliver1975} in $1975$. Such groups $G$ are characterized in the following way: $G$ does not contain subgroups $P,H\leq G$ fitting into a sequence $P\trianglelefteq H\trianglelefteq G$ such that $P$ and $G/H$ are groups of prime power order and the quotient $H/P$ is cyclic. We refer to such groups $G$ as \emph{Oliver groups}. The first examples of groups admitting one fixed point $G$-actions go back to Stein \cite{Stein1977} for $G=\SL(2,5)$ and Petrie \cite{Petrie1982} for abelian groups with three non-cyclic Sylow subgroups. In $1995$, Laitinen, Morimoto and Pawałowski \cite{Laitinen1995} proved that each nonsolvable group admits a one fixed point action on a sphere. Later, in $1998$, by modifying the techniques descibed in the paper concerning nonsolvable groups, Laitinen and Morimoto \cite{Morimoto1998} generalized this result and showed that Oliver groups coincide with the groups admitting one fixed point actions on spheres.

Further research has shown that the minimal dimension of a sphere admitting one fixed point actions is $6$. First, the combined work of De Michelis \cite{Demichelis1989}, Buchdahl et. al. \cite{Buchdahl1990} and Furuta \cite{Furuta1989} has shown that $S^n$ does not admit one fixed point actions for $n\leq 5$. Later, on the other hand, Bak and Morimoto \cite{Bak1992,Bak2005,Morimoto1987,Morimoto1989,Morimoto1991} proved that $A_5$, the alternating group on five letters, admits one fixed point actions on $S^n$ for every $n\geq 6$.

This leads us to the following question: given an Oliver group $G$ and an integer $n\geq 6$, is it true that there exists a one fixed point action of $G$ on $S^n$? Apart from the results described in the last paragraph concerning this question, other have been obtained as well. Let us discuss them. In $1977$, Stein \cite{Stein1977} proved that $\SL(2,5)$ can act on $S^7$ with exactly on fixed point. On the other hand, in $2016$, Borowiecka \cite{Borowiecka2016} excluded effective one fixed point actions of $\SL(2,5)$ on $S^8$. Two years later, in the joint work of Borowiecka and the author \cite{Borowiecka2018}, new exclusions of effective one fixed point actions have been obtained for Oliver groups of order up to $216$ and spheres of dimensions varying from $6$ to $10$. In $2020$, Morimoto and Tamura \cite{Morimoto2020} excluded odd fixed point actions of $S_5$, the symmetric group on five letters, and $\SL(2,5)$ on spheres of dimensions varying from $0$ to $13$. Further exclusion results concerning odd fixed point actions have been recently obtained by Tamura \cite{Tamura2020}. These results concern six Oliver groups (including symmetric groups, special linear and projective special linear groups, automorphism groups and Mathieu groups) and cover many of the dimensions from the set $\{0,...,50\}$. The methods used in the aticles of Borowiecka, Morimoto, Tamura and the author have been used to extend exclusion results in the PhD thesis of the author \cite[Theorem 1.4, Theorem 5.29]{Mizerka2020} by the application of GAP \cite{GAP4} software.

A general scheme in the study of transformation groups is to consider prescibed isotropy subgroups. The simpliest case constitute semifree actions (i. e. actions with trivial isotropy subgroups outside of the fixed point set). One can relax this criterion however and allow a weaker version of semifree actions. Such actions have restrictive dimensions of fixed point sets. After Illman \cite{Illman1982}, we call an action of a group $G$ on a smooth manifold $M$ a $k$-\emph{pseudofree} action if $\dim M^H\leq k$ for any nontrivial subgroup $H\leq G$. Lately, Morimoto announced a result that $3$-pseudofree one fixed point actions of $G$ on even-dimensional homotopy spheres $\Sigma^n$ are possible only for $G=A_5, A_5\times C_2$ or $G=S_5$ and $n=6$. On the other hand, if a $k$-pseudofree action is required to exist for arbitrarily large spheres, the group in question must possess a free representation. The classification of such groups was provided in \cite[6.3.1 Theorem]{Wolf1967}. Following this classification, there exists a group of order $240$ possessing a free representation which contains $\SL(2,5)$ as a subgroup of index $2$ (this group is denoted by $\TL_2(5)$ in \cite{Laitinen1986}). In the GAP \cite{GAP4} SmallGroup library, the group $\SL(2,5).C_2$ we consider appears with $\IdGroup=[240,89]$. It turns out that $\SL(2,5).C_2$ possesses a free representation (see Proposition \ref{proposition:primeCyclic}) and therefore $\SL(2,5).C_2\cong \TL_2(5)$. 

In a private correspondence with the author, Morimoto suggested that there may exist $6$-pseudofree effective one fixed point actions of $\SL(2,5).C_2$ on $S^n$ whenever $n=6+8k$ for some $k\geq 1$. Our main results concern the converse statement. Before, we formulate these results, let us recall a notion of a homology sphere. Given a commutative ring $R$ with unity, we call a closed smooth manifold $\Sigma$ of dimension $n$ an $R$-\emph{homology sphere} if the homology groups of $\Sigma$ with coefficients in $R$ coincide with the groups $H_k(S^n;R)$ for $k=0,\ldots,n$. The two main theorems of the paper are the following.
 \begin{theorem}\emph{[cf. Theorem \ref{theorem:main12}]}\label{theorem:main1}
 $\SL(2,5).C_2$ cannot act effectively with odd number of fixed points on $n$-dimensional $\mathbb{Z}$-homology sphere provided $n\in\{0,1,\ldots,13\}$. Moreover, $\SL(2,5).C_2$ cannot act effectively with exactly one fixed point on an $S^n$ provided $n\in\{0,1,\ldots,13\}\cup\{15,16,17\}$.
 \end{theorem}
 \begin{theorem}\emph{[cf. Theorem \ref{theorem:main22}]}\label{theorem:main2}
 There are no $5$-pseudofree one fixed point actions of $\SL(2,5).C_2$ on $\mathbb{Z}$-homology spheres. In case $\SL(2,5).C_2$ acts $6$-pseudofreely with exactly one fixed point on an $n$-dimensional $\mathbb{Z}$-homology sphere $\Sigma$, then $n=6+8k$ or $n=18+8k$ for $k\geq 0$. Moreover, if the action in question is effective and $n=6+8k$, then $k\geq 1$.
 \end{theorem}
 The article is organized as follows. In the first section we provide the necessary algebraic data for $\SL(2,5).C_2$. We compute the fixed point dimensions for irreducible $\mathbb{R}\SL(2,5).C_2$-modules for specific subgroups of $\SL(2,5).C_2$ and prove useful facts about these subgroups by means of the subgroup lattice of $\SL(2,5).C_2$. The next section contains the main technical exclusion results which the proofs of theorems \ref{theorem:main1} and \ref{theorem:main2} rely on. In the last section we prove the two main theorems of this paper.
\section{The necessary algebraic data for $\SL(2,5).C_2$}\label{section:algData}
In this section, we show the existence of subgroups of $\SL(2,5).C_2$ with suitable fixed point dimensions and forming generating sets for $\SL(2,5).C_2$. For that purpose, we present the real irreducible characters of $\SL(2,5).C_2$ and the fixed point dimensions for the corresponding representations of the actions of subgroups of our interest. The poset of conjugacy classes of $\SL(2,5).C_2$ shall be important as it provides the necessary information concerning the subgroup lattice of $\SL(2,5).C_2$.
\subsection{Nontrivial real irreducible characters of $\SL(2,5).C_2$}
The centre, $Z$, of $\SL(2,5).C_2$ is isomorphic to $C_2$ and $\SL(2,5).C_2/Z\cong S_5$. Among nontrivial real irreducible characters of $\SL(2,5).C_2$, $6$ characters are not faithful and $5$ are faithful. Let $U_1$, $U_{4,1}$, $U_{4,2}$, $U_{5,2}$, $U_{5,1}$ and $U_6$ be nontrivial irreducible $\mathbb{R}\SL(2,5).C_2$-modules with the non-faithful characters (these modules are faithful as $\mathbb{R}(\SL(2,5).C_2/Z=S_5)$-modules), and $W_{8,1}$, $W_{8,2}$, $W_{8,3}$, $W_{12,1}$, and $W_{12,2}$ irreducible faithful $\mathbb{R}\SL(2,5).C_2$-modules. The characters are presented in the table below (they can be derived from complex irreducible characters of $\SL(2,5).C_2$ which can be found at \cite{GroupNames} for example). The notation of conjugacy classes in the table is as follows. The class is denoted by a number indicating the order of its representative. If there is more than one class whose representative has a given order, then subsequent capital letters are used to distinguish between such classes.
\begin{table}[H]
$$
\begin{array}{c|rrrrrrrrrrrr}
  \rm class&\rm1&\rm2&\rm(3)&\rm(4A)&\rm(4B)&\rm(5)&\rm(6)&\rm(8A)&\rm(8B)&\rm(10)&\rm(12A)&\rm(12B)\cr
  \rm size&1&1&20&20&30&24&20&30&30&24&20&20\cr
\hline
  U_{1}&1&1&1&-1&1&1&1&-1&-1&1&-1&-1\cr
  U_{4,1}&4&4&1&-2&0&-1&1&0&0&-1&1&1\cr
  U_{4,2}&4&4&1&2&0&-1&1&0&0&-1&-1&-1\cr
   U_{5,1}&5&5&-1&-1&1&0&-1&1&1&0&-1&-1\cr
  U_{5,2}&5&5&-1&1&1&0&-1&-1&-1&0&1&1\cr
  U_{6}&6&6&0&0&-2&1&0&0&0&1&0&0\cr
  W_{8,1}&8&-8&-4&0&0&-2&4&0&0&2&0&0\cr
  W_{8,2}&8&-8&2&0&0&-2&-2&0&0&2&2\sqrt{3}&-2\sqrt{3}\cr
  W_{8,3}&8&-8&2&0&0&-2&-2&0&0&2&-2\sqrt{3}&2\sqrt{3}\cr
  W_{12,1}&12&-12&0&0&0&2&0&-2\sqrt{2}&2\sqrt{2}&-2&0&0\cr
  W_{12,2}&12&-12&0&0&0&2&0&2\sqrt{2}&-2\sqrt{2}&-2&0&0\cr
\end{array}
$$
\caption{\label{table:characters}Character table of $\SL(2,5).C_2$.}
\end{table}
\subsection{Fixed point dimensions for nontrivial irreducible $\mathbb{R}\SL(2,5).C_2$-modules}\label{subsection:dims}
The conjugacy classes of subgroups of $\SL(2,5).C_2$ are described in the poset below (see \cite{GroupNames}). The classes are denoted by their representatives (either by a well-known alias or by the id from the GAP \cite{GAP4} $\SmallGroup$ library). If there is more than one class with isomorphic representatives, then we add an appropriate letter to distinguish between these classes. The normal subgroups are marked by blue color.
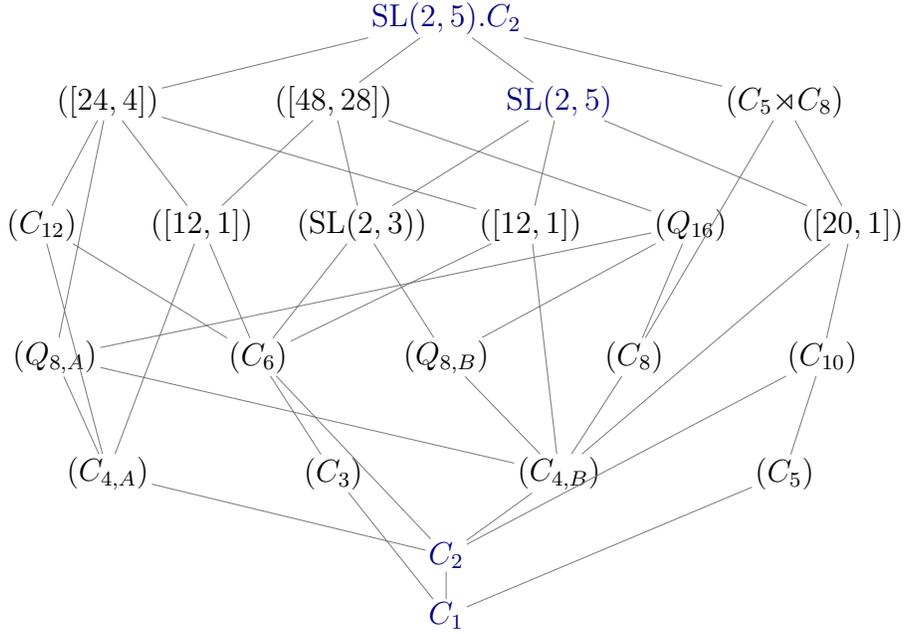
\begin{figure}[H]
$$
\begin{tikzpicture}[scale=1.0,sgplattice]
  \node[char] at (5.5,0) (1) {\gn{C1}{C_1}};
  \node[char] at (5.5,0.803) (2) {\gn{C2}{C_2}};
  \node at (4,1.89) (3) {\gn{C3}{(C_3)}};
  \node at (10,1.89) (4) {\gn{C5}{(C_5)}};
  \node at (1,1.89) (5) {\gn{C4}{(C_{4,A})}};
  \node at (7,1.89) (6) {\gn{C4}{(C_{4,B})}};
  \node at (3,3.44) (7) {\gn{C6}{(C_6)}};
  \node at (10.5,3.44) (8) {\gn{C10}{(C_{10})}};
  \node at (5.5,3.44) (9) {\gn{Q8}{(Q_{8,B})}};
  \node at (0.3,3.44) (10) {\gn{Q8}{(Q_{8,A})}};
  \node at (8,3.44) (11) {\gn{C8}{(C_8)}};
  \node at (2.25,5.2) (12) {\gn{Dic3}{([12,1])}};
  \node at (6.62,5.2) (13) {\gn{Dic3}{([12,1])}};
  \node at (0.125,5.2) (14) {\gn{C12}{(C_{12})}};
  \node at (10.9,5.2) (15) {\gn{Dic5}{([20,1])}};
  \node at (8.75,5.2) (16) {\gn{Q16}{(Q_{16})}};
  \node at (4.38,5.2) (17) {\gn{SL(2,3)}{(\SL(2,3))}};
  \node at (1,6.84) (18) {\gn{Dic6}{([24,4])}};
  \node at (10,6.84) (19) {\gn{C5:C8}{(C_5{\rtimes}C_8})};
  \node at (4,6.84) (20) {\gn{CSU(2,3)}{([48,28])}};
  \node[char] at (7,6.84) (21) {\gn{SL(2,5)}{\SL(2,5)}};
  \node[char] at (5.5,7.94) (22) {\gn{CSU(2,5)}{\SL(2,5).C_2}};
  \draw[lin] (1)--(2) (1)--(3) (1)--(4) (2)--(5) (2)--(6) (2)--(7) (3)--(7)
     (2)--(8) (4)--(8) (6)--(9) (5)--(10) (6)--(10) (6)--(11) (5)--(12)
     (7)--(12) (6)--(13) (7)--(13) (5)--(14) (7)--(14) (6)--(15) (8)--(15)
     (9)--(16) (10)--(16) (11)--(16) (7)--(17) (9)--(17) (12)--(18) (13)--(18)
     (14)--(18) (10)--(18) (15)--(19) (11)--(19) (12)--(20) (16)--(20) (17)--(20)
     (13)--(21) (15)--(21) (17)--(21) (18)--(22) (19)--(22) (20)--(22) (21)--(22);

\end{tikzpicture}
$$
\caption{\label{poset:characters}Subgroups of $\SL(2,5).C_2$.}
\end{figure}
Let us show first that $\SL(2,5).C_2$ is indeed isomorphic to the group $\TL_2(5)$ mentioned in \cite[Remark on p. 154]{Laitinen1986}. To this aim, we have to show that $\SL(2,5).C_2$ possesses a free representation. It turns out that the real irreducible $\SL(2,5).C_2$-representation $W_{8,1}$ is free. This follows from the proposition below, since the only primes dividing $240$, which is the order of $\SL(2,5).C_2$, are $2$, $3$, and $5$.
\begin{proposition}\label{proposition:primeCyclic}
Suppose $\mathcal{U}_4=\{U_{4,k}| k=1,2\}$, $\mathcal{U}_5=\{U_{5,k}|k=1,2\}$, $\mathcal{W}_8=\{W_{8,k}|k=1,2,3\}$, $\mathcal{W}_{12}=\{W_{12,k}|k=1,2\}$. Then, for any $U_4\in\mathcal{U}_4$, $U_5\in\mathcal{U}_5$, $W_8\in\mathcal{W}_8$, $W_{12}\in\mathcal{W}_{12}$, $W\in\mathcal{W}_8\cup \mathcal{W}_{12}$, and $i\in\{2,3\}$, the following equalities hold.
$$
\dim W^{C_2}=\dim W_{8,1}^{C_3}=\dim W_{8}^{C_5}=\dim W_{12}^{C_2}=0, 
$$
$$
\dim U_{1}^{C_2}=\dim U_1^{C_3}=\dim U_1^{C_5}=\dim U_{5}^{C_3}=\dim U_{5}^{C_5}=1,
$$
$$
\dim U_{4}^{C_3}=\dim U_6^{C_3}=\dim U_6^{C_5}=2,
$$
$$
\dim U_{4}^{C_2}=\dim W_{8,i}^{C_3}=\dim W_{12}^{C_3}=\dim W_{12}^{C_5}=4,
$$
$$
\dim U_{5}^{C_2}=5,\text{ and }\dim U_6^{C_2}=6.
$$
\end{proposition}
\begin{proof}
The equalities concerning the subgroup $C_2\leq \SL(2,5).C_2$ are straightforward as $\Ker U=C_2$ for $U=U_k$ for $k=1,4,5,6$ and $W$ is a faithful representation.

Since there are unique conjugacy classes of elements of order $3$, $5$, we get the desired equalities from Table \ref{table:characters} by using the formula for the fixed point dimension,
$$
\dim V^H=\frac{1}{|H|}\sum_{h\in H}\mychi_V(h),
$$
where $H\leq G$, $V$ is an $\mathbb{R}G$-module and $\mychi_V$ is its character.
\end{proof}
The next proposition provides fixed point dimensions for $4$ conjugacy classes of subgroups which are of our particular interest.
\begin{proposition}\label{proposition:mainSubgroupsDims}
Assuming the notations from Proposition \ref{proposition:primeCyclic}, and $H\in(C_{4,A})\cup(Q_{8,A})\cup(Q_{16})\cup([24,4])$, the following equalities hold.
$$
\dim U_1^H=\dim W^H=\dim U_{4,1}^{Q_{8,A}}=\dim U_{4,1}^{Q_{16}}=\dim U_{4,1}^{[24,4]}=\dim U_{5,1}^{[24,4]}=\dim U_6^{Q_{16}}=\dim U_6^{[24,4]}=0,
$$
$$
\dim U_{4,1}^{C_{4,A}}=\dim U_{4,2}^{Q_{16}}=\dim U_{4,2}^{[24,4]}=\dim U_{5,1}^{Q_{8,A}}=\dim U_{5,1}^{Q_{16}}=
\dim U_{5,2}^{Q_{16}}=\dim U_{5,2}^{[24,4]}=\dim U_6^{Q_{8,A}}=1,
$$
$$
\dim U_{4,2}^{Q_{8,A}}=\dim U_{5,1}^{C_{4,A}}=\dim U_{5,2}^{Q_{8,A}}=2,\text{ and }\dim U_{4,2}^{C_{4,A}}=\dim U_{5,2}^{C_{4,A}}=\dim U_6^{C_{4,A}}=3.
$$
\end{proposition}
\begin{proof}
The equality $\dim W^H=0$ follows immediately from Proposition \ref{proposition:primeCyclic}. 

Let $C_{4,A}'$ be the cyclic subgroup of order $4$, generated by an element from the conjugacy class $(4A)$. It follows from Table \ref{table:characters} that
$$
1>\frac{1}{4}\cdot(1-1+1+1)\geq \dim U_1^{C_{4,A}'}\geq 0.
$$
Thus $\dim U_1^{C_{4,A}'}=0$. On the other hand,
$$
1\geq \dim U_1^{\SL(2,3)}\geq\frac{1}{24}\cdot(1+1+8\cdot 1+6\cdot(-1)+8\cdot 1)>0,
$$
as $\SL(2,3)$ contains precisely: one element of order $1$ and $2$, eight elements of order $3$, six elements of order $4$, and eight elements of order $6$. Hence, $\dim U_1^{\SL(2,3)}=1$, which in connection with $\dim U_1^{C_{4,A}'}=0$ and Figure \ref{poset:characters} yields $C_{4,A}'\in(C_{4,A})$. Hence $\dim U_1^H=0$ as $H$ contains as a subgroup a representative of the class $(C_{4,A})$. Moreover, as we know now that the elements of $C_{4,A}$ of order $2$ belong to $(4A)$, all equalities involving $C_{4,A}$ follow directly from Table \ref{table:characters} and we skip it.

The subgroup $Q_{8,A}$ is isomorphic to the quaternion group, and therefore contains one element of order $1$ and $2$, and six elements of order $4$ (moreover, at least two of them are contained in the conjugacy class of elements $(4A)$ since a representative of $(C_{4,A})$ is contained in $(Q_{8,A})$). Thus,
$$
\dim U_{4,1}^{Q_{8,A}}\leq \frac{1}{8}\cdot (4+4+2\cdot(-2)+4\cdot 0)<1,
$$
which yields $\dim U_{4,1}^{Q_{8,A}}=\dim U_{4,1}^{Q_{16}}=\dim U_{4,1}^{[24,4]}=0$. The subgroup $Q_{16}$ contains one element of order $1$ and $2$, ten elements of order $4$, and four elements of order $8$, while $[24,4]$ contains one element of order $1$ and $2$, two elements of order $3$, fourteen elements of order $4$, two elements of order $6$, and four elements of order $12$. Using this knowledge, The equalities concerning $U_{5,2}$ can be derived directly from Table \ref{table:characters}, just as in the proof of Proposition \ref{proposition:primeCyclic}. We skip these computations. Moreover,
$$
\dim U_{5,1}^{[24,4]}<\frac{1}{24}\cdot(5+5+2\cdot (-1)+14\cdot 1+2\cdot(-1)+4\cdot (-1))<1,
$$
$$
\dim U_6^{Q_{16}}<\frac{1}{16}\cdot(6+6+10\cdot 0+4\cdot 0)<1,
$$
and
$$
\dim U_6^{[24,4]}<\frac{1}{24}\cdot(6+6+2\cdot 0+14\cdot 0+2\cdot 0+4\cdot 0)<1,
$$
which means that $\dim U_{5,1}^{[24,4]}=\dim U_6^{Q_{16}}=\dim U_6^{[24,4]}=0$.

Further proof is the repetition of the already used arguments. We are left with the proofs of equalities concerning dimensions which are to be equal $1$ or $2$. The following inequalities conclude then the proof.
$$
2>\frac{1}{16}\cdot(4+4+10\cdot 2+4\cdot 0)\geq \dim U_{4,2}^{Q_{16}}\geq\frac{1}{16}\cdot(4+4+10\cdot 0+4\cdot 0)>0.
$$
$$
2>\frac{1}{24}\cdot(4+4+2\cdot 1+14\cdot 2+2\cdot 1+4\cdot(-1))\geq\dim U_{4,2}^{[24,4]}\geq\frac{1}{24}\cdot(4+4+2\cdot 1+14\cdot 0+2\cdot 1+4\cdot(-1))>0,
$$
$$
2>\frac{1}{8}\cdot(5+5+2\cdot(-1)+4\cdot 1)\geq \dim U_{5,1}^{Q_{8,A}}\geq \frac{1}{8}\cdot(5+5+6\cdot(-1))>0,
$$
$$
1=\dim U_{5,1}^{Q_{8,A}}\geq \dim U_{5,1}^{Q_{16}}\geq\frac{1}{16}\cdot(5+5+10\cdot(-1)+4\cdot 1)>0,
$$
$$
2>\frac{1}{8}\cdot(6+6+6\cdot 0)\geq\dim U_6^{Q_{8,A}}\geq\frac{1}{8}\cdot(6+6+2\cdot 0+4\cdot(-2))>0,
$$
$$
3>\frac{1}{8}\cdot(4+4+6\cdot 2)\geq \dim U_{4,2}^{Q_{8,A}}\geq\frac{1}{8}\cdot(4+4+2\cdot 2+4\cdot 0)>1.
$$
\end{proof}
Another important fact is that the fixed point set, $U_1^{\SL(2,5)}$, has positive dimension. The group $\SL(2,5)$ contains one element of order $1$ and $2$, twenty elements of order $3$, thirty elements of order $4$, twenty four elements of order $5$, twenty elements of order $6$, and twenty four elements of order $10$. Hence,
$$
\dim U_1^{\SL(2,5)}\geq\frac{1}{120}\cdot(1+1+20\cdot 1+30\cdot(-1)+24\cdot 1+20\cdot 1+24\cdot 1)>0,
$$
which means that $\dim U_1^{\SL(2,5)}=1$.
\subsection{Suitable generating subgroups of $\SL(2,5).C_2$}
Recall that we denoted the centre of $\SL(2,5).C_2$ by $Z$. In the following two lemmas, let $\pi:\SL(2,5).C_2\rightarrow \SL(2,5).C_2/Z\cong S_5$ be the quotient epimorphism. As there shall be no confusion, from now on, we identify $\SL(2,5).C_2/Z$ with $S_5$, the group of all permutations on $5$ letters. The following figure is the poset of conjugacy classes of subgroups of $S_5$. As in the case of Figure \ref{poset:characters}, we mark the normal subgroups by blue color. In the figure below, $D_{2n}$ denotes the dihedral group of order $2n$, and $F_5$ is the Frobenius group of order $20$. This poset can be also found at \cite{GroupNames}. The convention for notation is the same as in Figure \ref{poset:characters}.
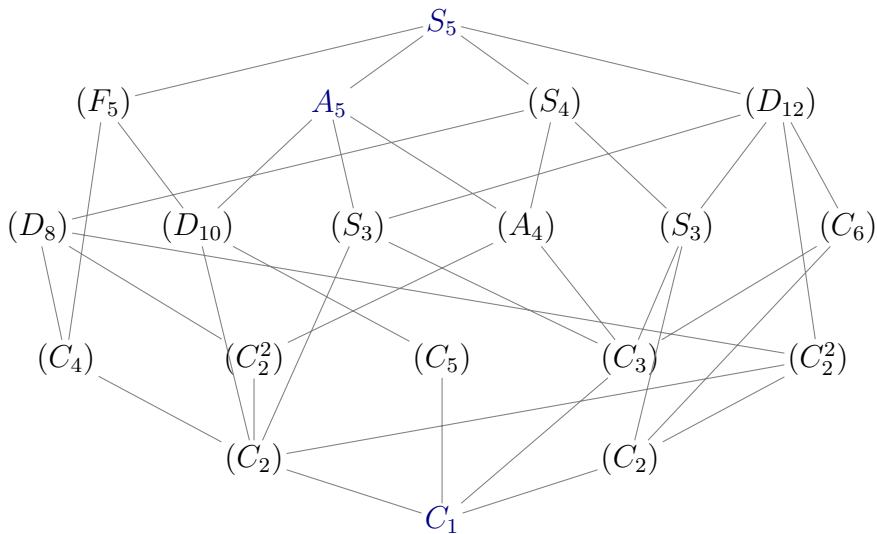
\begin{figure}[H]
	$$
	\begin{tikzpicture}[scale=1.0,sgplattice]
	\node[char] at (5.5,0) (1) {\gn{C1}{C_1}};
	\node at (8,0.803) (2) {\gn{C2}{(C_2)}};
	\node at (3,0.803) (3) {\gn{C2}{(C_2)}};
	\node at (8,2.14) (4) {\gn{C3}{(C_3)}};
	\node at (5.5,2.14) (5) {\gn{C5}{(C_5)}};
	\node at (3,2.14) (6) {\gn{C2^2}{(C_2^2)}};
	\node at (10.5,2.14) (7) {\gn{C2^2}{(C_2^2)}};
	\node at (0.5,2.14) (8) {\gn{C4}{(C_4)}};
	\node at (10.9,3.89) (9) {\gn{C6}{(C_6)}};
	\node at (8.75,3.89) (10) {\gn{S3}{(S_3)}};
	\node at (4.38,3.89) (11) {\gn{S3}{(S_3)}};
	\node at (2.25,3.89) (12) {\gn{D5}{(D_{10})}};
	\node at (0.125,3.89) (13) {\gn{D4}{(D_8)}};
	\node at (6.62,3.89) (14) {\gn{A4}{(A_4)}};
	\node at (10,5.54) (15) {\gn{D6}{(D_{12})}};
	\node at (1,5.54) (16) {\gn{F5}{(F_5)}};
	\node at (7,5.54) (17) {\gn{S4}{(S_4)}};
	\node[char] at (4,5.54) (18) {\gn{A5}{A_5}};
	\node[char] at (5.5,6.62) (19) {\gn{S5}{S_5}};
	\draw[lin] (1)--(2) (1)--(3) (1)--(4) (1)--(5) (3)--(6) (2)--(7) (3)--(7)
	(3)--(8) (2)--(9) (4)--(9) (2)--(10) (4)--(10) (3)--(11) (4)--(11)
	(3)--(12) (5)--(12) (6)--(13) (7)--(13) (8)--(13) (4)--(14) (6)--(14)
	(7)--(15) (9)--(15) (10)--(15) (11)--(15) (12)--(16) (8)--(16) (13)--(17)
	(14)--(17) (10)--(17) (12)--(18) (14)--(18) (11)--(18) (15)--(19) (16)--(19)
	(17)--(19) (18)--(19);

	\end{tikzpicture}
	$$
	\caption{\label{poset:S5}Subgroups of $S_5$.}
\end{figure}

Define the subgroups $L,K_1,K_2\leq S_5$ as follows.
$$
L=\langle(\begin{matrix}
1 &  3 
\end{matrix})\rangle\cong C_2\text{, }K_1=\langle (\begin{matrix}
1 & 3
\end{matrix}),(\begin{matrix}
4 & 5
\end{matrix})\rangle\cong C_2^2\text{ and }K_2=\langle (\begin{matrix}
1 & 2 & 3 & 4
\end{matrix}),(\begin{matrix}
1 & 2 
\end{matrix})(\begin{matrix}
3 & 4
\end{matrix})\rangle.
$$
\begin{lemma}\label{lemma:appropriateGroups}
	We have $K_2\cong D_8$ and $L=K_1\cap K_2$. 
\end{lemma}
\begin{proof}
	Let $D_8=\langle a,b|a^4=b^2=1,bab=a^{-1}\rangle$. Define $\varphi:D_8\rightarrow K_2$ by $\varphi(a)=(\begin{matrix}
	1 & 2&3&4 
	\end{matrix})$ and $\varphi(b)=(\begin{matrix}
	1 & 2
	\end{matrix})(\begin{matrix}
	3 & 4
	\end{matrix})$. Then
	$$
	\varphi(a^4)=\varphi(b^2)=1
	$$
	and
	$$
	\varphi(bab)=\varphi(b)\varphi(a)\varphi(b)=(\begin{matrix}
	1 & 2
	\end{matrix})(\begin{matrix}
	3&4 
	\end{matrix})(\begin{matrix}
	1 & 2&3&4 
	\end{matrix})(\begin{matrix}
	1 & 2 
	\end{matrix})(\begin{matrix}
	3&4 
	\end{matrix})=(\begin{matrix}
	1 & 4&3&2 
	\end{matrix})=\varphi(a^{-1}).
	$$
	Therefore $\varphi$ is a well-defined homomorphism. As it maps generators to generators, it is an epimorphism. Now, we have
	$$
	\Ker\varphi=\{b^{\varepsilon}a^k|\varphi(b^{\varepsilon}a^k)=\id\}=\{b^{\varepsilon}a^k|((\begin{matrix}
	1&2 
	\end{matrix})(\begin{matrix}
	3&4 
	\end{matrix}))^{\varepsilon}(\begin{matrix}
	1&2&3&4 
	\end{matrix})^k=\id\}.
	$$
	Notice that $(\begin{matrix}
	1&2&3&4 
	\end{matrix})^k=\id$ iff $4|k$. Thus, if $\varepsilon=0$, then $\varphi(b^{\varepsilon}a^k)=\id$ if and only if $b^{\varepsilon}a^k=1$. In case $\varepsilon=1$, we have
	$$
	\varphi(b^{\varepsilon})=\varphi(b)=(\begin{matrix}
	1&2 
	\end{matrix})(\begin{matrix}
	3&4 
	\end{matrix}),
	$$
	$$
	\varphi(b^{\varepsilon}a)=\varphi(ba)=(\begin{matrix}
	1&2 
	\end{matrix})(\begin{matrix}
	3&4 
	\end{matrix})(\begin{matrix}
	1&2&3&4 
	\end{matrix})=(\begin{matrix}
	2&4 
	\end{matrix}),
	$$
	$$
	\varphi(b^{\varepsilon}a^2)=\varphi(ba^2)=\varphi(ba)\varphi(a)=(\begin{matrix}
	2&4 
	\end{matrix})(\begin{matrix}
	1&2&3&4 
	\end{matrix})=(\begin{matrix}
	1&4 
	\end{matrix})(\begin{matrix}
	2&3 
	\end{matrix}),
	$$
	$$
	\varphi(b^{\varepsilon}a^3)=\varphi(ba^3)=\varphi(ba^2)\varphi(a)=(\begin{matrix}
	1&4 
	\end{matrix})(\begin{matrix}
	2&3 
	\end{matrix})(\begin{matrix}
	1&2&3&4 
	\end{matrix})=(\begin{matrix}
	1&3 
	\end{matrix}).
	$$
	Thus $\varphi(b^{\varepsilon}a^k)=\id$ if and only if $b^{\varepsilon}a^k=1$, whence $\varphi$ has the trivial kernel. Therefore $\varphi$ is an isomorphism and $K_2\cong D_8$.
	
	We show now that $L=K_1\cap K_2$. Notice that 
	$$
	(\begin{matrix}
	1 & 2 & 3 & 4
	\end{matrix})(\begin{matrix}
	1 & 2 
	\end{matrix})(\begin{matrix}
	3 & 4
	\end{matrix})=(\begin{matrix}
	1 & 3
	\end{matrix}).
	$$
	Thus $L\leq K_1\cap K_2$. This shows (as $|K_1|=4$) that $|K_1\cap K_2|\in\{2,4\}$, whence $K_1\cap K_2=\langle (\begin{matrix}
	1 & 3
	\end{matrix})\rangle$ or $K_1\cap K_2=K_1$. On the other hand, $ (\begin{matrix}
	4 & 5
	\end{matrix})\in K_1\setminus K_2$, whence $K_1\cap K_2=\langle (\begin{matrix}
	1 & 3
	\end{matrix})\rangle=L$.
\end{proof}
\begin{lemma}\label{lemma:apropriateGroups2}
	Using the notations from Figure \ref{poset:characters},
	$$
	\pi^{-1}(L)\in (C_{4,A})\text{, }\pi^{-1}(K_1)\in(Q_{8,A})\text{, and }  \pi^{-1}(K_2)\in(Q_{16}).
	$$
	Moreover,
	$$
	\pi^{-1}(L)=\pi^{-1}(K_1)\cap\pi^{-1}(K_2)\text{ and }\langle\pi^{-1}(K_1),\pi^{-1}(K_2)\rangle=\SL(2,5).C_2.
	$$ 
\end{lemma}
\begin{proof}
	Let us show first that $\pi^{-1}(L)\in(C_{4,A})$. We just have to exclude the case $\pi^{-1}(L)\in(C_{4,B})$. To this aim, we study the behaviour of subgroups from the class $(C_{4,B})$. Take $M=\langle g\rangle\in(C_{4,B})$ and let $V=U_{5,1}$. By Table \ref{table:characters}, $\dim V^M=3$. Therefore we can arrange the basis of $V$ to be $\mathcal{B}=\{v_1,\ldots,v_5\}$ such that $v_i$ remains fixed under the action of $g$ for $i=1,2,3$. Thus, the action of $g$ on $V$ is the order $4$ orthogonal transformation a $2$-dimensional real vector space. Hence, $g$ preserves the orientation of $V$ given by the basis $\mathcal{B}$. Note that $\pi(g)$ is the generator of $\pi(M)\cong C_2$. Therefore, $\pi(g)$ preserves the orientation of $V$ as well since $V$ can be viewed as an $\mathbb{R}S_5$-module as $\SL(2,5).C_2/\Ker(V)\cong S_5$. We show that $\pi(g)\in A_5$ which would mean, by Figure \ref{poset:characters} and Figure \ref{poset:S5}, that $\pi^{-1}(L)\notin(C_{4,B})$. We see by Figure \ref{poset:S5} that there are two conjugacy classes of elements of order $2$ in $S_5$. One of these conjugacy classes consists of a single element contained in $A_5$, while the other class has empty intersection with $A_5$. Take $\sigma$ which is the representative of the second class. Thus, $\sigma=\pi(g_{\sigma})$ for some $g_{\sigma}\in (4A)$, and $\mychi_V(\sigma)=-1$ by Table \ref{table:characters}. This means that the eigenvalues of the matrix determined by $\sigma$ are $1,1,-1,-1,-1$, and therefore $\sigma$ reverses the orientation. As $\pi(g)$ preserves the orientation, it follows then that $\pi(g)\in A_5$.  
	
	Since the conjugacy class $(Q_{16})$ from Figure \ref{poset:characters} is the unique conjugacy class of subgroups of $\SL(2,5).C_2$ containing subgroups of order $16$, and $|\pi^{-1}(K_2)|=16$ (as $|K_2|=8$ by Lemma \ref{lemma:appropriateGroups}), it follows that $\pi^{-1}(K_2)\in(Q_{16})$. Moreover, since $(Q_{8,A})$ is the only class with a representative of order $8$ and containing $C_{4,A}$, it follows that $\pi^{-1}(K_1)\in(Q_{8,A})$.
	
	In order to prove the second part, take any $g_1\in\pi^{-1}(\{(\begin{matrix}
	4&5
	\end{matrix})\})$ and $g_2\in\pi^{-1}(\{(\begin{matrix}
	1&2&3 & 4
	\end{matrix})\})$. Then $$\pi(g_2g_1)=\pi(g_2)\pi(g_1)=(\begin{matrix}
	1&2&3 & 4
	\end{matrix})(\begin{matrix}
	4 & 5
	\end{matrix})=(\begin{matrix}
	1&2&3&4 & 5
	\end{matrix}).$$
	Thus, the order of the element $g_2g_1$ is $10$, and we conclude by Figure \ref{poset:characters} that $\langle\pi^{-1}(K_1),\pi^{-1}(K_2)\rangle=\SL(2,5).C_2$. Finally, by Lemma \ref{lemma:appropriateGroups}, we have $\pi^{-1}(L)=\pi^{-1}(K_1\cap K_2)=\pi^{-1}(K_1)\cap \pi^{-1}(K_2)$.
\end{proof}
\section{The main exclusion results}\label{section:mainTech}
The following table contains useful information concerning fixed point dimensions for the four subgroups of $\SL(2,5).C_2$ considered in the previous section. This is a summary of Proposition \ref{proposition:mainSubgroupsDims}. In the table below, $i=2,3$ and $j=1,2$.
\begin{table}[H]
$$
\begin{tabular}{|c||c|c|c|c|c|c|c|c|c|c|c|c|c|c|c|c|c|c|c|c|c|c|}
\hline
     class&$(C_{4,A})$&$(Q_{8,A})$&$(Q_{16})$&$([24,4])$\\
     \hline
       \hline
     $U_1$&$0$&$0$&$0$&$0$\\
     \hline
     $U_{4,1}$&$1$&$0$&$0$&$0$\\
     \hline
     $U_{4,2}$&$3$&$2$&$1$&$1$\\
      \hline
     $U_{5,1}$&$2$&$1$&$1$&$0$\\
     \hline
     $U_{5,2}$&$3$&$2$&$1$&$1$\\
     \hline
     $U_6$&$3$&$1$&$0$&$0$\\
     \hline
     $W_{8,1}$&$0$&$0$&$0$&$0$\\
     \hline
     $W_{8,i}$&$0$&$0$&$0$&$0$\\
     \hline
     $W_{12,j}$&$0$&$0$&$0$&$0$\\
     \hline
\end{tabular}
$$
\caption{\label{table:dims}Fixed point dimension table for generating subgroups of $\SL(2,5).C_2$.}
\end{table}

In the following four lemmas, $G$ can be taken to be any finite group acting on a $\mathbb{Z}$-homology sphere $\Sigma$.
\begin{lemma}\emph{\cite[cf. Proposition 2.4]{Morimoto2020}}\label{lemma:MorimotoTamura}
     If $\Sigma^G$ is nonempty, then for any non-Oliver subgroup $H\leq G$ we have $\mychi(\Sigma^H)\neq 1$.
 \end{lemma}
\begin{lemma}\label{lemma:discreteStrategy}
 Suppose that $H_1$ and $H_2$ are non-Oliver subgroups of $G$ which generate $G$ and suppose $P$ is a prime power order subgroup of $H_1\cap H_2$. If there exists $x\in\Sigma^G$
	with $\dim T_x(\Sigma^P) = 0$, then $\Sigma^G$ is a two point set.
\end{lemma}
\begin{proof}
    By the Smith theory, it follows that $\Sigma^P$ is a $\mathbb{Z}_p$-homology sphere. As it is a finite set, we have $|\Sigma^P|=2$. Since $\Sigma^{H_1},\Sigma^{H_2}\subseteq\Sigma^P$, the Euler characteristics of $\Sigma^{H_1}$ and $\Sigma^{H_2}$ are equal to their cardinalities. On the other hand, by Lemma \ref{lemma:MorimotoTamura}, we conclude that $|\Sigma^{H_i}|\neq 1$ for $i=1,2$. Thus $|\Sigma^{H_i}|=2$ for $i=1,2$, and, in conclusion, $\Sigma^{H_1}=\Sigma^{H_2}=\Sigma^P$. As $\langle H_1,H_2\rangle=G$, we have $\Sigma^G=\Sigma^{H_1}\cap\Sigma^{H_2}=\Sigma^P$, and $\Sigma^G$ is therefore a two point set.
\end{proof}
The following lemma is a modified version of the previous one.
\begin{lemma}\label{lemma:discreteStrategy2}
	Suppose $\Sigma$ is a standard sphere and $H_1$, $H_2$ are non-Oliver subgroups generating $G$. Let $x_0\in\Sigma^G$ and put $V=T_{x_0}\Sigma$. Suppose $H\leq H_1\cap H_2$ is such that
	$$
	\dim V^K+\dim V^L=\dim V^H
	$$
	and $\Sigma^H\cong S^k$ for some $k\in\{0,1,2\}$. Then $\Sigma^G\neq\{x_0\}$.
\end{lemma}
\begin{proof}
	Suppose the assumptions of the Lemma hold. Consider the case $\dim V^{H_1}=\dim V^{H_2}=1$ as it is a bit different. Let $C(H_1)$ and $C(H_2)$ be the connected components containing $x_0$ of $\Sigma^{H_1}$ and $\Sigma^{H_2}$ respectively. Suppose for the converse that $\Sigma^G=\{x_0\}$. Then, as $\langle H_1,H_2\rangle=G$, we have $\Sigma^G=C(H_1)\cap C(H_2)$. On the other hand the intersection number of $C(H_1)$ and $C(H_2)$ in $\Sigma^H\cong S^k$ is zero. A contradiction.
	
Suppose now that at least one of the dimensions, $\dim V^{H_1}$ or $\dim V^{H_2}$, is not equal to $1$. As $\dim V^H=k$, where $k\in\{0,1,2\}$, this means that precisely one dimension from $\dim V^{H_1}$ and $\dim V^{H_2}$ is zero and the second one is $\dim V^H$. The case $k=0$ can be excluded by the previous lemma. Without loss of generality, suppose that $\dim V^{H_1}=\dim V^H$. This means that $\Sigma^{H_1}$ contains a connected component which is a closed submanifold of $\Sigma^H\cong S^k$ of dimension $k\in\{1,2\}$. Therefore this component must be entire $\Sigma^H$, which yields $\Sigma^{H_1}=\Sigma^H$. Thus, as $\Sigma^{H_2}\subseteq\Sigma^H$,
	$$
	\Sigma^G=\Sigma^{H_1}\cap\Sigma^{H_2}=\Sigma^H\cap\Sigma^{H_2}=\Sigma^{H_2}.
	$$
	Therefore, if $\Sigma^G=\{x_0\}$, then $\Sigma^{H_2}=\{x_0\}$. The latter, however, cannot hold since $H_2$ is a non-Oliver group and as such it cannot act on spheres with exactly one fixed point.
\end{proof}
\begin{lemma}\emph{\cite[Corollary 2.8]{Morimoto2020}}\label{lemma:index2MorimotoTamura}
Let $G_2$ be the intersection of all subgroups $H\leq G$ with $[G:H]\leq 2$ and suppose $x\in\Sigma^G$. If $\dim T_x(\Sigma)^{G_2}>0$, then $\Sigma^G\neq\{x\}$. 
\end{lemma}
The theorem below is similar to Theorem $5.1$ of \cite{Tamura2020}. 
\begin{theorem}\label{theorem:main}
Let $\Sigma$ be a $\mathbb{Z}$-homology sphere with $\SL(2,5).C_2$-action and $|\Sigma^{SL(2,5).C_2}|\equiv 1\Mod{2}$. Then the following statements hold.
\begin{enumerate}[(1)]
	\item If for each $x\in\Sigma^{\SL(2,5).C_2}$, $T_x(\Sigma)$ contains at least one of the modules $U_{4,2}$ or $U_{5,2}$, then the set of all points $x\in\Sigma^{\SL(2,5).C_2}$ such that $T_x(\Sigma)$ contains $U_6$ consists of odd number of points.
	\item If for at least one point $x\in\Sigma^{\SL(2,5).C_2}$, $T_x(\Sigma)$ contains neither of $U_{4,2}$ and $U_{5,2}$, then $T_x(\Sigma)$ contains $U_{5,1}$ or $U_6$.
\end{enumerate}
\end{theorem}
\begin{proof}
Let us prove first the first statement. Assume therefore that for each $x\in\Sigma^{\SL(2,5).C_2}$, $T_x(\Sigma)$ contains at least one of the modules $U_{4,2}$ or $U_{5,2}$. Denote by $F$ the set consisting of all points from $x\in\Sigma^{\SL(2,5).C_2}$ such that the tangential representation $T_x(\Sigma)$ contains $U_6$ and put $F'=\Sigma^{\SL(2,5).C_2}\setminus F$. We show that $|F'|\equiv 0\Mod{2}$.

If $F'$ is empty, then the assertion follows. Hence, we may assume $F'$ is non-empty and take $y\in F'$. Suppose now $F$ is non-empty and take $x\in F$. Then,
$$
T_x(\Sigma)=U_6^{\oplus k}\oplus U_1^{\oplus a}\oplus U_{4,1}^{\oplus b_1}\oplus U_{4,2}^{\oplus b_2}\oplus U_{5,1}^{\oplus c_1}\oplus U_{5,2}^{\oplus c_2}\oplus W,
$$
$$
T_y(\Sigma)=U_1^{\oplus r}\oplus U_{4,1}^{\oplus s_1}\oplus U_{4,2}^{\oplus s_2}\oplus U_{5,1}^{\oplus t_1}\oplus U_{5,2}^{\oplus t_2}\oplus W',
$$
where $W$ and $W'$ contain only the irreducibles denoted with the capital "W" and some index (these are precisely faithful real irreducibles of $\SL(2,5).C_2$) and $a,b_1,b_2,c_1,c_2,r,s_1,s_2,t_1,t_2$ are non-negative integers and $k$ is a positive integer. Moreover, by our assumption, we know that $s_2$ or $t_2$ is positive (and, similarly, $b_2$ or $c_2$ is positive).

Since $\langle Q_{16},[24,4]\rangle={\SL(2,5).C_2}$, we have
$
\Sigma^{Q_{16}}\cap\Sigma^{[24,4]}=\Sigma^{\SL(2,5).C_2}.
$
Note that $0<s_2+t_2=\dim T_y(\Sigma^{[24,4]})\leq \dim T_y(\Sigma^{Q_{8,A}})$, and $\dim T_y(\Sigma^{Q_{16}})=s_2+t_1+t_2>0$. Therefore both $\Sigma^{Q_{8,A}}$ and $\Sigma^{Q_{16}}$ are connected as positive dimensional $\mathbb{Z}_2$-homology spheres (by the Smith Theory) because $Q_{8,A}$ and $Q_{16}$ are $2$-groups. Summing up, by Table \ref{table:dims}, we get
$$
\dim\Sigma^{Q_{16}}=b_2+c_1+c_2, \dim\Sigma^{Q_{16}}=s_2+t_1+t_2,
$$
$$
\dim T_x(\Sigma^{[24,4]})=b_2+c_2, \dim T_y(\Sigma^{[24,4]})=s_2+t_2,
$$
and
$$
\dim\Sigma^{Q_{8,A}}=2b_2+c_1+2c_2+k, \dim\Sigma^{Q_{8,A}}=2s_2+t_1+2t_2.
$$
Thus,
$$
\dim\Sigma^{Q_{16}}+\dim T_x(\Sigma^{[24,4]})=\dim\Sigma^{Q_{8,A}}-k,
$$
$$
\dim\Sigma^{Q_{16}}+\dim T_y(\Sigma^{[24,4]})=\dim\Sigma^{Q_{8,A}}.
$$
Denote by $\Sigma^{[24,4]}_0$ the union of all connected components $C$ of $\Sigma^{[24,4]}$ such that $C\cap\Sigma^{\SL(2,5).C_2}\neq\emptyset$ and $\dim \Sigma^{Q_{16}}+\dim C=\dim\Sigma^{Q_{8,A}}$. Then $F\cap\Sigma^{[24,4]}_0=\emptyset$ and $F'=\Sigma^{[24,4]}_0\cap\Sigma^{\SL(2,5).C_2}$. Thus, the $\Mod{2}$-intersection number of $\Sigma^{Q_{16}}$ and $\Sigma^{[24,4]}_0$ in $\Sigma^{Q_{8,A}}$ is equal to $|F'|$ in $\mathbb{Z}_2$. On the other hand, the canonical $\Mod{2}$-intersection form on $\Sigma^{Q_{8,A}}$,
$$
\varphi:H_{s_2+t_1+t_2}(\Sigma^{Q_{8,A}};\mathbb{Z}_2)\times H_{s_2+t_1}(\Sigma^{Q_{8,A}};\mathbb{Z}_2)\rightarrow\mathbb{Z}_2
$$
is trivial as $\Sigma^{Q_{8,A}}$ is a $\mathbb{Z}_2$-homology sphere and $s_2+t_2>0$. Thus, the intersection number of $\Sigma^{Q_{16}}$ and $\Sigma^{[24,4]}_0$ in $\Sigma^{Q_{8,A}}$ is zero, whence $|F'|\equiv 0\Mod{2}$ in the case $F$ is non-empty. If $F$ is empty, then $\Sigma^{[24,4]}_0=\Sigma^{[24,4]}$ and the same argument with intersection numbers holds (in fact $F$ cannot be empty as by this argument, $|F'|$ is even which means that $|\Sigma^{\SL(2,5).C_2}|$ is even as well, as $\Sigma^{\SL(2,5).C_2}=F'$ in this case). This concludes the proof of the first statement.

Let us prove the second statement. Suppose therefore that for some $z\in\Sigma^{\SL(2,5).C_2}$, $T_z(\Sigma)$ contains neither of $U_{4,2}$ and $U_{5,2}$. If $T_z(\Sigma)$ contains $U_{5,1}$ or $U_6$, there is nothing to prove. Thus, we can assume that $T_z(\Sigma)$ does not contain neither of $U_{4,2}$, $U_{5,2}$, $U_{5,1}$, and $U_6$. Hence, $T_z(\Sigma)\cong U_1^{\oplus r}\oplus U_{4,1}^{\oplus s_1}\oplus W$ for $W$ containing only direct summands from the set $\{W_{8,1},W_{8,2},W_{8,3},W_{12,1},W_{12,2}\}$. This means, however, that $\dim T_z(\Sigma^{Q_{8,A}})=0$. As $Q_{16}$ and $[24,4]$ are non-Oliver and generate $\SL(2,5).C_2$, we conclude by Lemma \ref{lemma:discreteStrategy} that $|\Sigma^{\SL(2,5).C_2}|=2$ which is a contradiction.
\end{proof}
\section{Proofs of Theorem \ref{theorem:main1} and Theorem \ref{theorem:main2}}
In the light of Lemma \ref{lemma:apropriateGroups2}, we can assume that $C_{4,A}\leq Q_{8,A}\cap Q_{16}$ and $\langle Q_{8,A},Q_{16}\rangle={\SL(2,5).C_2}$. 
\begin{theorem}\emph{[cf. Theorem \ref{theorem:main1}]}\label{theorem:main12}
	$\SL(2,5).C_2$ cannot act effectively with odd number of fixed points on $n$-dimensional $\mathbb{Z}$-homology sphere provided $n\in\{0,1,\ldots,13\}$. Moreover, $\SL(2,5).C_2$ cannot act effectively with exactly one fixed point on an $S^n$ provided $n\in\{0,1,\ldots,13\}\cup\{15,16,17\}$.
\end{theorem}
\begin{proof}
	Suppose $\SL(2,5).C_2$ acts on a $\mathbb{Z}$-homology $n$-sphere $\Sigma$ with odd number of fixed points. Suppose first that for each $x\in\Sigma^{\SL(2,5).C_2}$, $T_x(\Sigma)$ contains at least one of the modules $U_{4,2}$ or $U_{5,2}$. By Theorem \ref{theorem:main} we conclude that there exists $x\in\Sigma^{\SL(2,5).C_2}$ such that $T_x(\Sigma)$ contains $U_6$. As the action is faithful, $T_x(\Sigma)$ must contain one of the modules from the set $\{W_{8,1},W_{8,2},W_{8,3},W_{12,1},W_{12,2}\}$. This would mean that $n\geq 6+8=14$.
	
	Assume that for at least one point $x\in\Sigma^{\SL(2,5).C_2}$, $T_x(\Sigma)$ contains neither of $U_{4,2}$, $U_{5,2}$, and $U_6$. By Theorem \ref{theorem:main} we conclude that $T_x(\Sigma)$ contains $U_{5,1}$. Thus (by Table \ref{table:dims}), in order to prove the assertion concerning actions with odd number of fixed points, we have to exclude three possibilities: $T_x(\Sigma)\cong U_{5,1}\oplus U_{8,i}$ for $i=1,2,3$. Fortunately, we will be able to exclude these three possibilities at once - just by using the fact that $U_{5,1}$ appears as a unique non-faithful direct summand of $T_x(\Sigma)$. The proof of the first part can be finished by considering the $\Mod{2}$-intersection numbers of $\Sigma^{Q_{8,A}}$ and $\Sigma^{Q_{16}}$ in $\Sigma^{C_{4,A}}$ as $\Sigma^{Q_{8,A}}$, $\Sigma^{Q_{16}}$ and $\Sigma^{C_{4,A}}$ are connected by the Smith theory (because they have positive dimension by Table \ref{table:dims}) and $\dim\Sigma^{Q_{8,A}}+\dim \Sigma^{Q_{16}}=1+1=2=\dim\Sigma^{C_{4,A}}$. Indeed, the intersection number argument yields $|\Sigma^{\SL(2,5).C_2}|$ even as $\Sigma^{Q_{8,A}}\cap\Sigma^{Q_{16}}=\Sigma^{\SL(2,5).C_2}$ for $Q_{8,A}$ and $Q_{16}$ generate $\SL(2,5).C_2$. 
	
	Assume now that $\SL(2,5).C_2$ acts effectively with exactly one fixed point on an $S^n$ for $n\in\{15,16,17\}$. Denote the fixed point by $x$ and put $V=T_x\Sigma$. By Lemma \ref{lemma:index2MorimotoTamura} and the last paragraph of subsection \ref{subsection:dims}, we conclude that $V$ cannot contain $U_1$ as a direct summand. As $V$ contains an effective module as a direct summand, and $\dim V>13$ (which has been shown in the first part of the proof), it follows that $\dim V\geq 8+5+4=17$. On the other hand, by Theorem \ref{theorem:main}, $V$ must contain $U_{5,1}$ or $U_6$ as a direct summand. Thus, by the effectiveness of $V$ and Table \ref{table:dims}, one of the following possibilities must hold: $V\cong U_{4,1}\oplus U_{5,1}\oplus W_{8,i}$, $V\cong U_{4,2}\oplus U_{5,1}\oplus W_{8,i}$ or $V\cong U_{5,1}\oplus W_{12,j}$ for $i\in\{1,2,3\}$ and $j\in\{1,2\}$. If $V\cong U_{4,1}\oplus U_{5,1}\oplus W_{8,i}$, then, by Table \ref{table:dims}, we have $\dim V^{Q_{8,A}}=1=1+0=\dim V^{Q_{16}}+\dim V^{[24,4]}$ and we can exclude this case by Lemma \ref{lemma:discreteStrategy2}. If $V\cong U_{4,2}\oplus U_{5,1}\oplus W_{8,i}$, then $\dim V^{C_{4,A}}=5=3+2=\dim V^{Q_{8,A}}+\dim V^{Q_{16}}$, whereas if $V\cong U_{5,1}\oplus W_{12,j}$, then $\dim V^{C_{4,A}}=2=1+1=\dim V^{Q_{8,A}}+\dim V^{Q_{16}}$. Both situations can be excluded by the intersection number argument.
\end{proof}
\begin{theorem}\emph{[cf. Theorem \ref{theorem:main2}]}\label{theorem:main22}
There are no $5$-pseudofree one fixed point actions of $\SL(2,5).C_2$ on $\mathbb{Z}$-homology spheres. In the case $\SL(2,5).C_2$ acts $6$-pseudofreely with exactly one fixed point on an $n$-dimensional $\mathbb{Z}$-homology sphere $\Sigma$, then $n=6+8k$ or $n=18+8k$ for $k\geq 0$. Moreover, if the action in question is effective and $n=6+8k$, then $k\geq 1$.
\end{theorem}
\begin{proof}
Suppose first that $\SL(2,5).C_2$ acts $6$-pseudofreely with exactly one fixed point on some $\mathbb{Z}$-homology sphere $\Sigma$. Let $\Sigma^{\SL(2,5).C_2}=\{x\}$ and put $V=T_x(\Sigma)$. By Theorem \ref{theorem:main12}, it follows that $V$ contains either $U_{5,1}$ or $U_6$ as a direct summand. On the other hand, by Lemma \ref{lemma:index2MorimotoTamura}, $V$ cannot contain $U_1$ as a direct summand since, by the last paragraph of subsection \ref{subsection:dims}, $\dim U_1^{\SL(2,5)}=1$. As $\SL(2,5).C_2$ acts $6$-pseudofreely, it follows by Proposition \ref{proposition:primeCyclic} that $V\cong U_{5,1}\oplus W$ or $V\cong U_6\oplus W$, where $W$ contains only $W_{8,i}$ and $W_{12,j}$ as direct summands. The former case can be however excluded by the intersection number argument as $\dim V^{C_{4,A}}=2=1+1=\dim V^{Q_{8,A}}+\dim V^{Q_{16}}$. This shows immediately that $\SL(2,5).C_2$ cannot act $5$-pseudofreely. Moreover, if $\SL(2,5).C_2$ acts $6$-pseudofreely, then, by Proposiotion \ref{proposition:mainSubgroupsDims}, $V\cong U_6\oplus W$, where $W$ consists only of modules $W_{8,i}$ and at most one of the modules $W_{12,j}$. This shows $n=6+8k$ or $n=18+8k$ for $k\geq 0$. In case the action in question is effective and $n=6+8k$, $W$ must by nontrivial and thus $k\geq 1$. This concludes the proof.
\end{proof}
\section*{Acknowledgements}
 The author would like to thank Prof. Masaharu Morimoto for the suggestion to consider the problem presented in this article and many helpful comments during the work on it. I would like to thank also Prof. Krzysztof Pawałowski for important suggestions which essentially improved the presentation of this paper. My sincere thanks go as well to Mr. Shunsuke Tamura for enlightening discussions on the subject.


\bibliographystyle{acm}
\bibliography{sample}
\textcolor{white}{fsf}\\
\emph{Faculty of Mathematics and Computer Science}\\
\emph{Adam Mickiewicz University in Poznań}\\
\emph{ul Uniwersytetu Poznańskiego 4}\\
\emph{61-614 Poznań, Poland}\\
\emph{Email address:} piotr.mizerka@amu.edu.pl







\end{document}